\documentclass[11 pt]{amsart}
\usepackage{amssymb, amsmath, amsthm, amsfonts, amscd}
\usepackage{xcolor}
\usepackage{datetime}
\usepackage{enumerate}
\usepackage{mathabx}
\usepackage{blkarray}
\usepackage{hhline}

\newtheorem{theorem}{Theorem}[section]
\newtheorem{proposition}[theorem]{Proposition}
\newtheorem{lemma}[theorem]{Lemma}
\newtheorem{corollary}[theorem]{Corollary}

\theoremstyle{definition}
\newtheorem{definition}[theorem]{Definition}
\newtheorem{example}[theorem]{Example}

\theoremstyle{remark}
\newtheorem{remark}[theorem]{Remark}

\numberwithin{equation}{section}

\begin{document}
	
	\title[A representation of hyponormal $\mathcal{AN}$-operators]{A representation of hyponormal absolutely norm attaining operators}
	
	\author{Neeru Bala}
	\address{Department of Mathematics, Indian Institute of Technology - Hyderabad, Kandi, Sangareddy, Telangana, India 502 285.}
	\email{ma16resch11001@iith.ac.in}
	
	\author{Ramesh G.}
	\address{Department of Mathematics, Indian Institute of Technology - Hyderabad, Kandi, Sangareddy, Telangana, India 502 285.}
	\email{rameshg@iith.ac.in}
	\date{\today}
	\begin{abstract}
In this article, we characterize absolutely norm attaining normal operators in terms of the essential spectrum. Later we prove a structure theorem for hyponormal absolutely norm attaining (or $\mathcal{AN}$-operators in short) and deduce  conditions for the normality of the operator.
	\end{abstract}
\maketitle
\section{Introduction}
The class of hyponormal operators is an important class of non-normal operators. The problem "when is a hyponormal operator normal?" is studied by several researchers. It is known that a compact hyponormal operator is normal (see \cite{ANDOhy,berbariab1962} for more details). We refer \cite{MARTIN} for several other such sufficient conditions. One more important result in this direction is that a hyponormal operator whose spectrum has zero area measure is normal, which follows by Putnam's inequality \cite[Theorem 1]{PUTNAM}.

In this article, we consider similar questions by weakening the above mentioned conditions. More precisely  we replace the compact operator by an operator in the bigger class, namely $\mathcal{AN}$-class of operators. Another important result we prove is that a hyponormal $\mathcal{AN}$-operator whose essential spectrum and the Weyl spectrum coincide must be normal. To prove these results we first establish a representation of hyponormal $\mathcal{AN}$-operators.

Let $H$ be a complex Hilbert space and $T$ be a bounded linear operator on $H$. Then $T$ is called \textit{norm attaining} if there exist $x\in H$, $\|x\|=1$ such that $ \|Tx\|=\|T\|$,  and \textit{ absolutely norm attaining} or $\mathcal{AN}$-operator, if for any closed subspace $M$ of $H$, the operator $T|_M:M\rightarrow H$ is norm attaining, that is there exist $x\in M,\,\|x\|=1$ such that $\|T|_Mx\|=\|Tx\|=\|T|_M\|$. The set of absolutely norm attaining operators is a subclass of norm attaining operators and contains the space of all compact operators, isometries. In general, it contains partial isometries with finite-dimensional null space. This class was first studied by Carvajal and Neves in \cite{CARVAJAL }. Structure of positive absolutely norm attaining operators has been studied in \cite{PAL,RAMpara,VENKU}. In \cite{VENKU}, a characterization of normal and self-adjoint $\mathcal{AN}$-operators are studied. In general, the adjoint of an $\mathcal{AN}$-operator need not be $\mathcal{AN}$ (see \cite{CARVAJAL } for more details), but when the operator is normal, this is indeed true.

In this article, first we study a characterization of normal $\mathcal{AN}$-operators in terms of the essential spectrum. Later we prove a structure theorem for hyponormal $\mathcal{AN}$-operators and as a consequence, we deduce that a compact hyponormal operator is normal. In the end, we prove that if an $\mathcal{AN}$-operator and its adjoint are paranormal then the operator must be normal.

In the remaining part of this section, we introduce basic notions and notations used in the article. In the second section, we give a characterization of normal $\mathcal{AN}$-operators. In the last section, we prove a representation theorem for the hyponormal $\mathcal{AN}$-operators and deduce important consequences.

\subsection{Preliminaries}
We denote the space of all bounded linear operators from $H_1$ to $H_2$ by $\mathcal{B}(H_1,H_2)$. For $T\in\mathcal{B}(H_1,H_2)$, $R(T)$ and $N(T)$ denote the range and null space of $T$, respectively. If $R(T)$ is finite dimensional, then $T$ is called a \textit{finite-rank operator}. A bounded linear operator $T$ is said to be \textit{compact}, if $T$ maps every bounded set in $H_1$ into a pre-compact set in $H_2$. The set of all finite-rank operators and compact operators in $\mathcal{B}(H_1,H_2)$ are denoted by $\mathcal{F}(H_1,H_2)$ and $\mathcal{K}(H_1,H_2)$, respectively. In particular $\mathcal{F}(H):=\mathcal{F}(H,H)$ and $\mathcal{K}(H):=\mathcal{K}(H,H)$.

If $T\in\mathcal{B}(H_1,H_2)$, then the \textit{adjoint} operator $T^*:H_2\rightarrow H_1$ is a bounded linear operator satisfying
$\langle Tx,y\rangle=\langle x,T^*y\rangle,\,\forall\,x\in H_1,\,y\in H_2$.

For $T\in\mathcal{B}(H)$, $\rho(T)=\{\lambda\in\mathbb{C}:T-\lambda I\text{ is invertible in }\mathcal{B}(H)\}$ is called the \textit{resolvent} set of $T$ and $\sigma(T)=\mathbb{C}\setminus\rho(T)$ is called the \textit{spectrum} of $T$.

An operator $T\in\mathcal{B}(H)$ is said to be \textit{Fredholm}, if $R(T)$ is closed, $N(T)$ and $N(T^*)$ are finite dimensional. In this case, the index of $T$ is defined by $$ind(T)=\dim N(T)-\dim N(T^*).$$ The \textit{essential spectrum}, and the \textit{Weyl spectrum} of $T$ are defined by
\begin{align*}
\sigma_{ess}(T):=&\{\lambda\in\mathbb{C}:T-\lambda I\text{ is not Fredholm}\},\\
\omega(T):=&\{\lambda\in\mathbb{C}:T-\lambda I\text{ is not Fredholm of index } 0 \},
\end{align*}
respectively. We define  $$\pi_{00}(T)={\{\lambda \in \sigma(T): \lambda \; \text{is an isolated eigenvalue with finite multiplicity}}\}.$$ The quantity $m_e(T):=\inf\{\lambda:\lambda\in\sigma_{ess}(|T|)\}$ is called the \textit{essential minimum modulus} of $T$.
For more details about Fredholm theory and the essential spectrum, we refer to \cite{BAX,MULL,SCH}.

If $M$ is a closed subspace of $H$, then the unit sphere in $M$ is denoted by $S_M=\{x\in M:\|x\|=1\}$. For $T\in\mathcal{B}(H)$, $M$ is said to be invariant under $T$ if $TM\subseteq M$. For a non-negative real number $r$, the open disc and circle with centre $0$ and radius $r$ are denoted by $D(0,r):=\{z\in\mathbb{C}:|z|<r\}$ and $C(0,r):=\{z\in\mathbb{C}:|z|=r\}$, respectively.

For $T\in\mathcal{B}(H)$, $m(T):=\inf{\{\|Tx\|; x\in H_1,\; \|x\|=1}\}$ is called the \textit{minimum modulus} of $T$. For more details, we refer to \cite{CARVAJALNEVES, GAN}.

 Let $T\in\mathcal{B}(H)$. Then $T$ is said to be \textit{normal} if $T^*T=TT^*$, \textit{self-adjoint} if $T^*=T$ and \textit{positive} if $\langle Tx,x\rangle\geq 0,\,\forall x\in H.$
 \begin{definition}
 Let $T\in\mathcal{B}(H)$. Then $T$ is called
 \begin{enumerate}
 	\item \textit{hyponormal}, if $TT^*\leq T^*T$.
 	\item \textit{paranormal}, if $\|Tx\|^2\leq\|T^2x\|\|x\|,\,\forall\, x\in H$.
 \end{enumerate}
 \end{definition}
Note that every hyponormal operator is paranormal.
\begin{theorem}\cite{ANDOhy,berbariab1962}\label{compact hypo}
	Every compact hyponormal operator $T\in\mathcal{B}(H)$ is normal.
\end{theorem}
The above result is true for paranormal operators as well, see \cite[Theorem 2]{Istratescuetal} for more details.

\begin{lemma}\cite[Theorem 3]{Lee}\label{inv subspace}
	Suppose $T\in\mathcal{B}(H)$ is a paranormal norm attaining operator. Then $N(|T|-\|T\|I)$ is an invariant subspace for $T$.
\end{lemma}
An operator $T\in\mathcal{B}(H)$ is called \textit{absolutely norm attaining} if for every closed subspace $M\subseteq H$, $T|_M$ is norm attaining.  We denote the set of all absolutely norm attaining operators in $\mathcal{B}(H_1,H_2)$ by $\mathcal{AN}(H_1,H_2)$ and $\mathcal{AN}(H):=\mathcal{AN}(H,H)$.

First we recall a few important results related to positive $\mathcal{AN}$-operators which we need to prove our results. For positive absolutely norm attaining operators, we have the following characterizations.

\begin{theorem}\label{paulsenpandey}\cite[Theorem 5.1]{PAL}
	Let $H$ be a complex Hilbert space of arbitrary dimension and let $P$ be a positive operator on $H$. Then $P$ is an $\mathcal{AN}$-operator if and only if $P$ is of the form $P=\alpha I+K+F$, where $\alpha \geq 0$, $K$ is a positive compact operator and $F$ is a self-adjoint finite rank operator.
\end{theorem}
We remark that the representation in Theorem \ref{paulsenpandey} can be made unique and the value of $\alpha$ can be found out to be $m_e(T)$, the essential minimum modulus of $T$. For full details, we refer to \cite{VENKU}.

In \cite{RAMpara} the following representation for normal $\mathcal{AN}$-operators is proved.

\begin{theorem}\cite[Theorem 3.9, Theorem 3.13]{RAMpara}\label{normal char}
	Let $T\in\mathcal{AN}(H)$ be a normal operator. Then there exist $(H_{\beta},U_{\beta})_{\beta\in\sigma(|T|)}$ such that
	\begin{enumerate}
		\item $H_{\beta}$ is a reducing subspace for $T$,
		\item  $U_{\beta}\in\mathcal{B}(H_{\beta})$ is a unitary operator.
	\end{enumerate}
such that
\begin{enumerate}
	\item $H=\underset{\beta\in\sigma(|T|)}{\oplus}H_{\beta}$,
	\item $T=\underset{\beta\in\sigma(|T|)}{\oplus}\beta U_{\beta}$,
	\item $\sigma(T)\subseteq\underset{\beta\in\sigma(|T|)}{\cup}\beta\mathbb{T}$, where $\mathbb{T}=\{z\in\mathbb{C}:|z|=1\}$.
\end{enumerate}
\end{theorem}
In the above theorem, precisely $H_\beta=N(|T|-\beta I)$. The above theorem implies that every normal $\mathcal{AN}$-operator is direct sum of scalar multiple of unitary operators.

\begin{theorem}\cite[Theorem 2.4]{RAMpara}\label{esschar}
	Let $H$ be an infinite dimensional Hilbert space and $T\in\mathcal{B}(H)$ be positive. Then $T\in\mathcal{AN}(H)$ if and only if $\sigma_{ess}(T)$ is singleton set and $[m(T),m_e(T))$ contains at most finitely many points of $\sigma(T)$.
\end{theorem}
\section{Normal $\mathcal{AN}$-operators}

Here we give a spectral characterization of normal absolutely norm attaining operators in terms of the essential spectrum.
\begin{theorem}\label{normal AN}
	Let $H$ be an infinite-dimensional Hilbert space and $T\in\mathcal{B}(H)$ be a normal operator. Then the following are equivalent.
	\begin{enumerate}
		\item\label{AN} $T\in\mathcal{AN}(H)$.
		\item\label{Ball} There exists a non-negative real number $\alpha\geq 0$ such that $\sigma_{ess}(T)\subseteq C(0,\alpha)$ and $D(0,\alpha)$ contains at most finitely many points of $\sigma(T)$.
	\end{enumerate}
\end{theorem}
\begin{proof}
	(\ref{AN})$\Rightarrow$(\ref{Ball}) Let $T\in\mathcal{AN}(H)$. By Theorem \ref{normal char}, $T$ can be represented as
	\begin{equation*}
	T=\underset{\beta\in\sigma(|T|)}{\oplus}\beta U_{\beta},
	\end{equation*}
	where $U_{\beta}$ is unitary operator on $H_{\beta}=N(|T|-\beta I)$. Since $T\in\mathcal{AN}(H)$, we get $|T|\in\mathcal{AN}(H)$ by \cite[Lemma 6.2]{PAL} and consequently $\sigma_{ess}(|T|)$ is singleton set by Theorem \ref{esschar}. Thus $H_{\beta}$ is finite-dimensional subspace of $H$ for all $\beta\ne m_e(|T|)$ (that is $\beta\ne m_e(T)$). Let us write $\alpha=m_e(T)$. As $U_{\beta}$ is a finite rank operator for every $\beta\ne \alpha$, it follows that $\sigma_{ess}(T)\subseteq C(0,\alpha)$.
	
	Also, $\sigma(T)\subseteq \underset{\beta\in\sigma(|T|)}{\cup}\sigma(\beta U_{\beta})$. As $\sigma(U_{\beta})$ is finite for $\beta\ne \alpha$ and by Theorem \ref{esschar} there are at most finitely many $\beta\in\sigma(|T|)$ such that $\beta< \alpha$, say $\beta_1,\beta_2,\ldots\beta_m$ (see \cite[Theorem 3.9]{RAMpara} for details), which imply $$D(0, \alpha)\cap\sigma(T)=\underset{i=1}{\overset{m}{\cup}}[\sigma(\beta_iU_{\beta_i})\cap\sigma(T)]\subseteq \underset{i=1}{\overset{m}{\cup}}\beta_i\sigma(U_{\beta_i}).$$
	For $1\leq i\leq m$, $U_{\beta_i}$ is a finite-rank operator, which implies $\underset{i=1}{\overset{m}{\cup}}\beta_i\sigma(U_{\beta_i})$ is finite and consequently $D(0,\alpha)\cap\sigma(T)$ is a finite subset of $\mathbb{C}$.
	
	
	(\ref{Ball})$\Rightarrow$(\ref{AN}) Let $D(0,\alpha)\cap\sigma(T)=\{\gamma_1,\gamma_2,\ldots\gamma_{n_0}\}$ for $n_0\in\mathbb{N}\cup\{0\}$ and $|\gamma_i|=r_i$ for $1\leq i\leq n_0$.
	By the continuous functional calculus, we know that  $\lambda\in\sigma(T)$ implies $|\lambda|\in\sigma(|T|)$. Therefore $r_1,r_2,\ldots r_{n_0}\in\sigma(|T|)$. Also note that by definition $m_{e}(T)=\alpha$.
	
We claim that $\sigma(|T|)\cap[0,\alpha)=\{r_1,r_2,\ldots r_{n_0}\}$. If there exist some $r< \alpha$ such that $r\in\sigma(|T|)$, then again by the functional calculus $ C(0,r)\cap\sigma(T)\ne\emptyset$. So, $r$ should be one of these $r_i$, $1\leq i\leq n_0$.
	
	Finally, we have that $\sigma_{ess}(|T|)=\{ \alpha \}$ and $[0, \alpha)\cap\sigma(|T|)$ is a finite set, namely $\{r_1,r_2,\ldots r_{n_0}\}$. By Theorem \ref{esschar}, $|T|\in\mathcal{AN}(H)$ and consequently $T\in\mathcal{AN}(H)$ by \cite[Lemma 6.2]{PAL}.
\end{proof}
\begin{remark}
	For non-normal $\mathcal{AN}$ operators part \textit{(\ref{Ball})} of Theorem \ref{normal AN} need not hold. The following example illustrates this.
\end{remark}
\begin{example}\label{counternonnormal}
	Consider the right shift operator $R:\ell^2(\mathbb{N})\rightarrow \ell^2(\mathbb{N})$ defined by $R(x_1,x_2,\ldots)=(0,x_1,x_2,\ldots),\,\forall(x_n)\in \ell^2(\mathbb{N})$. As $R$ is an isometry, it is an $\mathcal{AN}$-operator. We know that $\sigma(R)=\overline{D(0,1)}$, $\sigma_{ess}(R)=C(0,1)$ and $m_e(R)=1$. It is easy to observe that $D(0,1)$ contains uncountably many points of $\sigma(R)$. Though $R\in\mathcal{AN}(\ell^2(\mathbb{N}))$, it does not satisfy part \textit{(\ref{Ball})} of Theorem \ref{normal AN}.
\end{example}
By \cite[Proposition 3.4]{VENKU}, we know that if $T=T^*\in \mathcal{AN}(H)$, then the essential spectrum of $T$ contains at most two points. Here, in fact by Theorem \ref{normal AN} we can characterize self-adjoint $\mathcal{AN}$-operators. A different characterization of self-adjoint $\mathcal{AN}$-operators is given in Theorem $3.3$ and Corollary $4.5$ of \cite{VENKU}.
\begin{corollary}
	Let $T\in\mathcal{B}(H)$ be a self-adjoint operator. Then $T\in\mathcal{AN}(H)$ if and only if there exists $\alpha \geq 0$ such that $\sigma_{ess}(T)\subseteq\{-\alpha,\alpha\}$ and $(-\alpha,\alpha)\cap \sigma(T)$ is at most finite.
\end{corollary}
\begin{proof}
	From Theorem \ref{normal AN}, $T\in\mathcal{AN}(H)$ if and only if there exists $\alpha \geq 0$ such that  $\sigma_{ess}(T)\subseteq \partial D(0,\alpha))$ and $D(0,\alpha)\cap\sigma(T)$ is at most finite. Since $T=T^*$, we have $\sigma(T)\subseteq[-\|T\|,\|T\|].$ In this case $D(0,\alpha)\cap\sigma(T)=(-\alpha,\alpha)\cap\sigma(T) $ is at most finite and $\sigma_{ess}(T)\subseteq\{-\alpha,\alpha\}.$
\end{proof}

We conclude this section with a comment on absolutely minimum attaining operators. Recall that $T\in \mathcal B(H_1,H_2)$ is called \textit{minimum attaining}, if there exists $x\in H_1$ with $\|x\|=1$ such that $\|Tx\|=m(T)$.
An operator $T$ is said to be \textit{absolutely minimum attaining} (or $\mathcal{AM}$-operator, in short)
if for every non zero closed subspace $M$ of $H$, the operator $T|_M:M\rightarrow H$ is minimum attaining. We refer \cite{CARVAJALNEVES,GAN,GAN1} and references therein for more detail about $\mathcal{AM}$-operators. The structure of multiplication $\mathcal{AM}$-operators as well as normal $\mathcal{AM}$-operators are discussed in \cite{RAM}.

By imitating the proof of Theorem \ref{normal AN} and using \cite[Theorem 4.4]{RAM}, we can get the following characterization for normal $\mathcal{AM}$-operators.
\begin{theorem}\label{normal AM}
	Let $T\in\mathcal{B}(H)$ be a normal operator. The following are equivalent.
	\begin{enumerate}
		\item\label{AM} $T\in\mathcal{AM}(H)$.
		\item\label{Ball1} There exists a real number $\beta\geq 0$ such that $Ann(0;\beta,\|T\|)$ contains at most finitely many points of $\sigma(T)$ and $\sigma_{ess}(T)\subseteq C(0,\beta)$, where $Ann(0;\beta,\|T\|):=\{z\in\mathbb{C}:\beta<|z|\leq\|T\|\}$.
	\end{enumerate}
\end{theorem}
\begin{remark}
	The spectral decomposition of positive $\mathcal{AM}$-operators is established in \cite{RAM}.
\end{remark}
\section{Hyponormal $\mathcal{AN}$-operators}
In this section, we prove a structure theorem for hyponormal $\mathcal{AN}$-operators and deduce a few important consequences.
\begin{lemma}\label{lemma red subspace of hyponoral op}
	Let $T\in\mathcal{B}(H)$ be a non-zero norm attaining hyponormal operator. If $N\left(|T|-\|T\|I\right)$ is finite dimensional, then it is a reducing subspace for $T$. Moreover $T|_{N(|T|-\|T\|I)}=\|T\|U$, where $U\in\mathcal{B}(N(|T|-\|T\|I))$ is a unitary operator.
\end{lemma}
\begin{proof}
	Let $M:=N\left(|T|-\|T\|I\right)$. By Lemma \ref{inv subspace}, we know that $M$ is an invariant subspace for $T$ and $\frac{T}{\|T\|}$ is an isometry on $M$. But as $M$ is finite dimensional, we have that $\frac{T}{\|T\|}$ is unitary on $M$, denote it by $U$. Hence $T|_M=\|T\|U$.
	
	Since $M$ is an invariant subspace for $T$, $T$ has the following matrix representation;
	\begin{equation*}
	T=\begin{bmatrix}
	T|_M&A\\
	0&B
	\end{bmatrix},
	\end{equation*}
	where $A\in\mathcal{B}(M^{\perp},M),\,B\in\mathcal{B}(M^{\perp})$. Since $T$ is hyponormal, we have
	\begin{align*}
	0\leq T^*T-TT^*=\begin{bmatrix}
	-AA^*&(T|_M)^*A-AB^*\\
	A^*T|_M-BA^*&A^*A+B^*B-BB^*
	\end{bmatrix}.
	\end{align*}
	By the positivity of operator matrix, we get $-AA^*\geq 0$, which imply $A=0$. Hence
	\begin{align}\label{eqn2}
	T=\begin{bmatrix}
	T|_M&0\\
	0&B
	\end{bmatrix}.
	\end{align}
	That is, $M$ is a reducing subspace for $T$.
\end{proof}
\begin{remark}
	Note that $B$ is hyponormal in the above representation (\ref{eqn2}). In fact, $T$ is normal if and only if $B$ is normal.
\end{remark}
\begin{proposition}\label{prop rep of hypo AN op}
	Let $T\in\mathcal{B}(H)$ be a non-zero hyponormal $\mathcal{AN}$-operator with $\sigma_{ess}(|T|)=\{\|T\|\}$. If $\pi_{00}(|T|)$ is non-empty and $\pi_{00}(|T|)=\{\lambda_i\}_{i=1}^{m_0}$ for some $m_0\in\mathbb{N}$, then
	\begin{enumerate}
		\item $H=H_1\oplus H_2$, where $H_1=N(|T|-\|T\|I)$, $H_2=\underset{i=1}{\overset{m_0}{\oplus}}N(|T|-\lambda_iI)$.
		\item $T$ has the following representation with respect to $H_1\oplus H_2$.
		\[T=
		\begin{blockarray}{ccc}
		H_1&H_2\\
		\begin{block}{(cc)c}
		\|T\|S_0&A&H_1\\
		0& B&H_2\\
		\end{block}
		\end{blockarray}
		\]
		where
		\begin{enumerate}
			\item $S_0\in\mathcal{B}(H_1)$ is an isomtery.
			\item $A,\,B$ are finite-rank operators with $S_0^*A=0$ and $(A+B)^*(A+B)=\underset{i=1}{\overset{m_0}{\oplus}}\lambda_i^2I_{N(|T|-\lambda_iI)}$. 
		\end{enumerate}
		If $\pi_{00}(T)$ is an empty set then $T=\|T\|S_0$.
	\end{enumerate}
\end{proposition}
\begin{proof}
	Without loss of generality, we assume that $T=S|T|$, where $S\in\mathcal{B}(H)$ is an isometry. To see this, let $T=V|T|$ be the polar decomposition of $T$, where $V$ is a partial isometry with the initial space $N(T)^{\perp}$ and the final space $\overline{R(T)}$. As $T$ is hyponormal we have that $N(T)\subseteq N(T^*)=R(T)^{\perp}$, so there exist a partial isometry $W$ with the initial space $N(T)$ and the final space $R(T)^{\perp}$. Then $S:=V+W$ is an isometry on $H$ and $T=S|T|$.
	
	Note that as $|T|\in\mathcal{AN}(H)$, we have $\|T\|\in\sigma_p(|T|)$. Since $\sigma_{ess}(|T|)=\{\|T\|\}$, it is clear that $N(|T|-\|T\|I)$ is infinite dimensional subspace of $H$. By \cite[Lemma 3.4]{RAMpara}, $N\left(|T|-\|T\|I\right)=N(T^*T-\|T\|^2I)$ is invariant under $T$. We observe from the following equation that $N\left(|T|-\|T\|I\right)$ is invariant under $S$;
	\begin{equation*}
	Sx=\frac{S|T|x}{\|T\|}=\frac{Tx}{\|T\|}\in N\left(|T|-\|T\|I\right),\,\forall x\in N\left(|T|-\|T\|I\right).
	\end{equation*}
	
	Since $|T|\in\mathcal{AN}(H)$, it follows that $\pi_{00}(|T|)$ is at most finite by Theorem \ref{esschar}. If $\pi_{00}(|T|)$ is empty set, then $H_2=\{0\}, H=H_1$ and $T=\|T\|S$.
	
	Now, we assume that $\pi_{00}(|T|)$ is non-empty and equal to $\{\lambda_i\}_{i=1}^{m_0}$ for some $m_0\in\mathbb{N}$, where $\lambda_i<\|T\|$.
	Thus we have
	\begin{equation}\label{matrix eqn}
	S=
	\begin{blockarray}{ccc}
	H_1&H_2\\
	\begin{block}{(cc)c}
	S|_{H_1}&S_1&H_1\\
	0&S_2&H_2\\
	\end{block}
	\end{blockarray}.
	\end{equation}
	Note that $S_1=P_{H_1}S|_{H_2}$ and $S_2=P_{H_2}S|_{H_2}$. Also
	\[
	|T|=
	\begin{bmatrix}
	\|T\|I_{H_1}&0\\
	0&F
	\end{bmatrix},
	\]
	where $F=\underset{i=1}{\overset{m_0}{\oplus}}\lambda_iI_{M_i}$ and $M_i=N(|T|-\lambda_iI)$ is finite dimensional subspace of $H$.
	From this we conclude that
	\[T=
	\begin{bmatrix}
	S|_{H_1}&S_1\\
	0&S_2
	\end{bmatrix}
	\begin{bmatrix}
	\|T\|I_{H_1}&0\\
	0&F
	\end{bmatrix}=
	\begin{bmatrix}
	\|T\|S|_{H_1}&S_1F\\
	0&S_2F
	\end{bmatrix}.\]

	Now we write $A:=S_1F$, $B:=S_2F$ and $S_0=S|_{H_1}$. As $T$ is hyponormal operator, we have
	\begin{align*}
	0&\leq T^*T-TT^*\\
	&=\begin{bmatrix}
	\|T\|^2I_{H_1}-\|T\|^2P_{R(S_0)}-AA^*&\|T\|S_0^*A-AB^*\\
	\|T\|A^*S_0-BA^*&A^*A+B^*B-BB^*
	\end{bmatrix}.
	\end{align*}
	Thus $AA^*\leq \|T\|^2P_{R(S_0)^{\perp}}$, which gives $\overline{R(A)}=\overline{R(AA^*)}\subseteq R(S_0)^{\perp}=N(S_0^*)$. Consequenlty $S_0^*A=0$. We have $A+B=(S_1+S_2)F$, thus $(A+B)^*(A+B)=F^2=\underset{i=1}{\overset{m_0}{\oplus}}\lambda_i^2I_{M_i}$, as $S_1^*S_1+S_2^*S_2=I_{H_2}$ by equation (\ref{matrix eqn}).
\end{proof}

\begin{corollary}\label{cor hypo}
	The operator $T$ in Proposition \ref{prop rep of hypo AN op} is normal if and only if $S_0$ is unitary.
\end{corollary}
\begin{proof}
	If $S_0$ is unitary then by the condition $S_0^*A=0$, we have $A=0$. In this case, $B$ is a hyponormal operator. As it is finite-rank operator, it must be normal by Theorem \ref{compact hypo}. Hence $T$ is normal.
	
	Conversely assume that $T$ is a normal operator. Then $N(|T|-\|T\|I)=N(T^*T-\|T\|^2I)=N(TT^*-\|T\|^2I)$ is a reducing subspace for $T$, that is $H_1$ is a reducing subspace for $T$. Let $T=S|T|$, where $S$ is as in Proposition \ref{prop rep of hypo AN op}. Since $T$ is normal, we can easily prove that $S|T|=|T|S$. From this and the fact that $H_1$ reduces $T$, we can conclude that $H_1$ is reducing for $S$ also. Hence we have
	\[
	S=
	\begin{bmatrix}
	S_0&0\\
	0&S_2
	\end{bmatrix}.
	\]
	As a result $A=0\cdot F=0$ and the normality of $T$ forces that $S_0$ to be unitary.
\end{proof}
\begin{theorem}\label{thm hypo char}
	Let $T\in\mathcal{AN}(H)$ be a hyponormal operator. Then
	\begin{enumerate}
		\item there exists closed subspaces $H_0,H_1,H_2$ with $\dim(H_2)<\infty$ such that $H=H_0\oplus H_1\oplus H_2$.
		\item $T$ has the following representation with respect to $H_0\oplus H_1\oplus H_2$;
		\[T=
		\begin{blockarray}{cccc}
		H_0&H_1&H_2\\
		\begin{block}{(ccc)c}
		S_0&0&0&H_0\\
		0&\lambda S_1&A&H_1\\
		0&0&S_2&H_2\\
		\end{block}
		\end{blockarray},
		\]
		where $\sigma_{ess}(|T|)=\{\lambda\}$, for some $\lambda\in\mathbb{R}$ and
		\begin{enumerate}
			\item $S_0=\underset{i=1}{\overset{n_0}{\oplus}}\lambda_iU_i$ for some unitary operator $U_i\in\mathcal{B}\left(N\left(|T|-\lambda_iI\right)\right)$, if $(\lambda,\|T\|]\cap\sigma(|T|)=\{\lambda_i\}_{i=1}^{n_0}$ for some $n_0\in\mathbb{N}\cup\{\infty\}$ or $S_0=0$ if $(\lambda,\|T\|]\cap\sigma(|T|)$ is empty set.
			\item $S_1\in\mathcal{B}(H_1)$ is an isometry.
			\item $A,\,S_2$ are finite-rank operators with $S_1^*A=0$ and  $$(A+S_2)^*(A+S_2)=\underset{j=1}{\overset{m_0}{\oplus}}\delta_j^2I_{N(|T|-\delta_jI)},$$ if $[0,\lambda)\cap\sigma(|T|)=\{\delta_j:1\leq j\leq m_0\}$ for some $m_0\in\mathbb{N}$.
		\end{enumerate}
	\end{enumerate}
\end{theorem}
\begin{proof}
	Let $T\in\mathcal{AN}(H)$. Then $|T|\in\mathcal{AN}(H)$ and hence $\sigma_{ess}(|T|)$ is singleton, say $\{\lambda\}$ and $\sigma(|T|)$ is countable, by Theorem \ref{esschar}. Again using Theorem \ref{esschar}, we know that $[0,\lambda)$ contains at most finitely many points of $\sigma(|T|)$ and $(\lambda, \|T\|]$ contains at most countably many points of $\sigma(|T|)$.
	
	First, we assume that $\lambda<\|T\|$ and consequently $(\lambda, \|T\|]\cap\sigma(|T|)$ is non-empty and $(\lambda, \|T\|]\cap\sigma(|T|)=\{\lambda_i\}_{i=1}^{n_0}$ for some $n_0\in\mathbb{N}\cup\{\infty\}$. Now, define $H_0=\underset{i=1}{\overset{n_0}{\oplus}}N\left(|T|-\lambda_iI\right)$ and $H_1=N(|T|-\lambda I)$.
	
	Without loss of generality, we assume that $\lambda_1=\|T\|$. By Lemma \ref{lemma red subspace of hyponoral op}, $N(|T|-\|T\|I)$ is reducing subspace for $T$ and $T|_{N(|T|-\|T\|I)}=\|T\|U$ for some unitary operator $U\in\mathcal{B}(N(|T|-\|T\|I))$. Again $T|_{\left(N(|T|-\|T\|I)\right)^{\perp}}$ is hyponormal $\mathcal{AN}$-operator. Now repeating the same procedure, we get  $S_0:=T|_{H_0}=\underset{i=1}{\overset{n_0}{\oplus}}\lambda_i U_i$, where $U_i\in\mathcal{B}(N(|T|-\lambda_i I))$ is unitary operator for every $1\leq i\leq n_0$.
	
	As $H_0$ is a reducing subspace for $T$, we get $T|_{H_0^{\perp}}$ is a hyponormal $\mathcal{AN}$-operator and $\|T|_{H_0^{\perp}}\|\leq\lambda$. Now, it is enough to look at the structure of $T|_{H_0^\perp}$.
	
	We consider the following four cases which exhaust all possibilities.
	
	Case $(1)$: $\lambda$ is an eigenvalue of $|T|$ with infinite multiplicity but not a limit point of $\sigma(|T|)$: \\
	In this case $(\lambda,\|T\|]\cap\sigma(|T|)$ is a finite set by \cite[Theorem 3.8]{PAL} and consequently $n_0\in\mathbb{N}$. In this case $\|T|_{H_0^{\perp}}\|=|\lambda|$ and $\{\lambda\}=\sigma_{ess}(T|_{H_0^{\perp}})$. Using Proposition \ref{prop rep of hypo AN op}, we get $H_0^{\perp}=H_1\oplus H_2$, where $H_2$ is a finite dimensional space and $H_1=N(|T|-\lambda I)$. Hence we have
	\[
	T|_{H_0^{\perp}}=
	\begin{bmatrix}
	\lambda S_1&A\\
	0&S_2
	\end{bmatrix},
	\]
	where $S_1,\,A$ and $S_2$ satisfy conditions $(b)$ and $(c)$. Therefore $T$ has the following representation satisfying $(a),\,(b),\,(c)$.
	\[
	T=
	\begin{blockarray}{cccc}
	H_0&H_1&H_2\\
	\begin{block}{(ccc)c}
	S_0&0&0&H_0\\
	0&\lambda S_1&A&H_1\\
	0&0&S_2&H_2\\
	\end{block}
	\end{blockarray}.
	\]
	
	Case $(2)$: $\lambda$ is not an eigenvalue of $|T|$ but it is a limit point of $\sigma(|T|)$:\\
	In this case, $(\lambda,\|T\|]\cap\sigma(|T|)=\{\lambda_1,\lambda_2,\lambda_3,\ldots\}$ is an infinite set by \cite[Theorem 3.8]{PAL}, and $H_1=\{0\}$. As $\lambda$ is not an eigenvalue of $T$, we get $\|T|_{H_0^\perp}\|<\lambda$. If $\sigma(|T|)\cap[0,\lambda)$ is empty, then $H_2=\{0\}$ and $T=S_0$.
	
	If $[0,\lambda)\cap\sigma(|T|)$ is non-empty and is equal to $\{\delta_j\}_{j=1}^{m_0}$ for some $m_0\in\mathbb{N}$, then take $H_2=\underset{i=1}{\overset{m_0}{\oplus}}N(|T|-\delta_i I)$ and
	\[
	T=
	\begin{bmatrix}
	S_0&0\\
	0&S_2
	\end{bmatrix},
	\]
	where $S_2=\underset{j=1}{\overset{m_0}{\oplus}}\delta_jV_j$ and $V_j\in\mathcal{B}\left(N(|T|-\delta_jI)\right)$ is a unitary operator.
	
	Case $(3)$: $\lambda$ is neither an eigenvalue of $|T|$ nor a limit point of $\sigma(|T|)$:\\
	Then $(\lambda,\|T\|]\cap\sigma(|T|)$ is a finite set by \cite[Theorem 3.8]{PAL}, thus $n_0\in\mathbb{N}$ and $H_1=\{0\}$.
	
	 If $[0,\lambda)\cap\sigma(|T|)$ is empty set, then $H=H_0$ and $T=S_0$.
	
	If $[0,\lambda)\cap\sigma(|T|)$ is non-empty and equal to $\{\delta_1,\delta_2,\ldots\delta_{m_0}\}$ for some $m_0\in\mathbb{N}$, then $H=H_0\oplus H_2$ is finite dimensional Hilbert space, where $H_2=\underset{i=1}{\overset{m_0}{\oplus}}N(|T|-\delta_i I)$. As $H_0$ is a reducing subspace for $T$, thus we have
	\[T=
	\begin{bmatrix}
	T|_{H_0}&0\\
	0&T|_{H_0^{\perp}}
	\end{bmatrix}.
	\]
	Both $T|_{H_0}$ and $T|_{H_0^{\perp}}$ are hyponormal finite rank operators, hence normal. By Theorem \ref{normal char}
	$$T|_{H_0}=\underset{i=1}{\overset{n_0}{\oplus}}\lambda_iU_i\text{ and }T|_{H_0^{\perp}}=\underset{j=1}{\overset{m_0}{\oplus}}\delta_jV_j,$$
	where $U_i\in\mathcal{B}\left(N(|T|-\lambda_iI)\right)$ and $V_j\in\mathcal{B}\left(N(|T|-\delta_jI)\right)$ are unitary operators.
	
	Case $(4))$ $\lambda$ is an eigenvalue of $|T|$ also it is a limit point of $\sigma(|T|)$:\\
	Here $(\lambda,\|T\|]\cap\sigma(|T|)$ is an infinite set, say $\{\lambda_1,\lambda_2,\ldots\}$ by \cite[Theorem 3.8]{PAL}. In this case $\|T|_{H_0^{\perp}}\|=\lambda$.
	
	If $[0,\lambda)\cap\sigma(|T|)$ is non-empty and equal to $\{\delta_1,\delta_2,\ldots\delta_{m_0}\}$ for some $m_0\in\mathbb{N}$, then take $H_2:=\underset{i=1}{\overset{m_0}{\oplus}}N(|T|-\delta_i I)$ and we have
	\[
	T=
	\begin{bmatrix}
	S_0&0&0\\
	0&\lambda S_1&A\\
	0&0&S_2
	\end{bmatrix}.
	\]
	Using Lemma \ref{lemma red subspace of hyponoral op} and Proposition \ref{prop rep of hypo AN op}, we get either $S_1$ is unitary or isometry depending on the multiplicity of the eigenvalue $\lambda$. Note that $A=0$ if $\lambda$ is of finite multiplicity but it need not be the zero operator, if $\lambda$ is of infinite multiplicity.
	
	If $[0,\lambda)\cap\sigma(|T|)$ is empty set, then $H_2=\{0\}$ and
	\[T=
	\begin{bmatrix}
	S_0&0\\
	0&\lambda S_1
	\end{bmatrix}.
	\]
	
	If $\lambda=\|T\|$, then the result follows from Proposition \ref{prop rep of hypo AN op}
\end{proof}
\begin{corollary}\label{ANhypo implies normal}
	Let $T\in\mathcal{B}(H)$ satisfy the assumptions of Theorem \ref{thm hypo char}. Then $T$ is normal if and only if $S_1$ is a unitary operator.
\end{corollary}
\begin{proof}
	Let $S_1$ be a unitary operator. As $S_1^*A=0$ and $S_1$ is unitary, we get $A=0$. In this case, $S_2$ is a hyponormal operator. As it is finite-rank operator, it must be normal by Theorem \ref{compact hypo}. Hence $T$ is normal.
	
	Conversely, suppose that $T$ is a normal operator. As $H_1$ is a reducing subspace for $T$, this implies $T|_{H_1^{\perp}}$ is normal and satisfy the assumptions of Corollary \ref{cor hypo}. Hence $S_1$ is a unitary operator and $A=0$.
\end{proof}
\begin{remark}
	Let $T\in\mathcal{AN}(H)$ be a hyponormal operator. From Theorem \ref{thm hypo char}, we observe that $\sigma(T)\subseteq \overline{D(0,\lambda)\cup\left(\underset{i=1}{\overset{n_0}{\cup}}C(0,\lambda_i)\right)}$, where $\sigma_{ess}(|T|)=\{\lambda\}$ and $(\lambda,\|T\|]\cap\sigma(|T|)=\{\lambda_1,\lambda_2,\ldots\lambda_{n_0}\}$, for some $n_0\in\mathbb{N}\cup\{\infty\}.$
\end{remark}
Now, we demonstrate Theorem \ref{thm hypo char} with an example.
\begin{example}
	Define $T:\ell^2(\mathbb{N})\rightarrow\ell^2(\mathbb{N})$ by
	\begin{align*}
	T(x_1,x_2,\ldots)=\left(\frac{x_1}{2},0,\frac{x_2}{2},x_3,x_4,\ldots\right),\,\forall\,(x_1,x_2,\ldots)\in\ell^2(\mathbb{N}).
	\end{align*}
	Then
	\begin{align*}
	T^*(x_1,x_2,\ldots)=\left(\frac{x_1}{2},\frac{x_3}{2},x_4,\ldots\right),\,\forall\,(x_1,x_2,\ldots)\in\ell^2(\mathbb{N}).
	\end{align*}
	Clearly $\|T^*(x_n)\|\leq\|T(x_n)\|,\forall\,(x_n)\in\ell^2(\mathbb{N})$, hence $T$ is hyponormal. By simple computation, we get
	\begin{align*}
	T^*T(x_1,x_2,\ldots)=\left(\frac{x_1}{4},\frac{x_2}{4},x_3,\ldots\right).
	\end{align*}
	It is easy to see that $T^*T=I-F$, where $$F(x_n)=\left(\frac{3x_1}{4},\frac{3x_2}{4},0\ldots\right),\,\forall \,(x_n)\in\ell^2(\mathbb{N}).$$
	By Theorem \ref{paulsenpandey}, we conclude that $T^*T$ is $\mathcal{AN}$ and hence $T$ is $\mathcal{AN}$ by \cite[Corollary 2.11]{VENKU}. Now choose $H_1:=\overline{\text{span}}\{e_3,e_4,\ldots\}$ and $H_2=\text{span}\{e_1,e_2\}$, then $T$ can be represented as
	\[T=
	\begin{blockarray}{cccc}
	H_1&\hspace{1 cm} H_2\\
	\begin{block}{(c|cc)c}
	R&\hspace{1 cm}A&&H_1\\
	\hhline{---}0&\frac{1}{2}&0\\
	0&0&0&\vspace{0.3 cm}H_2\\
	\end{block}
	\end{blockarray}
	\]
	where $R(x_3,x_4,\ldots)=(0,x_3,x_4,\ldots)$ and $A(x_1,x_2)=(x_2/2,0,\ldots)$.
\end{example}
The following examples illustrates the fact that the Theorem \ref{thm hypo char} can be improved in terms of the unitary operators for some particular cases but not in general.
\begin{example}
Define $T_1,T_2:\ell^2(\mathbb{N})\rightarrow\ell^2(\mathbb{N})$ by
\begin{align*}
T_1(x_1,x_2,x_3,x_4,\ldots)=&(0,x_1,x_2,2x_3,2x_4,\ldots),\,\forall\,(x_n)\in\ell^2(\mathbb{N})\\
T_2(x_1,x_2,x_3,x_4,\ldots)=&(x_1,2x_2,0,3x_3,3x_4,\ldots),\,\forall\,(x_n)\in\ell^2(\mathbb{N}).
\end{align*}
Let $\{e_n\}_{n\in\mathbb{N}}$ be the standard orthonormal basis for $\ell^2(\mathbb{N})$. For $H_1:=\overline{\text{span}}\{e_3,e_4,\ldots\}$ and $H_2=\text{ span }\{e_1,e_2\}$, we have
\[T_1=
\begin{blockarray}{ccc}
H_1&H_2\\
\begin{block}{(cc)c}
2R_1&A_1&H_1\\
0&B_1&H_2\\
\end{block}
\end{blockarray}
\text{ and }
T_2=
\begin{blockarray}{ccc}
H_1&H_2\\
\begin{block}{(cc)c}
3R_1&0&H_1\\
0&B_2&H_2\\
\end{block}
\end{blockarray}
\]
where
\begin{align*}
R_1(x_3,x_4,\ldots)=&(0,x_3,x_4,\ldots),\,\forall\,(x_3,x_4,\ldots)\in H_1,\\
A_1(x_1,x_2)=&(x_2,0,0,\ldots),\,\forall\,(x_1,x_2)\in H_2,\\
B_1(x_1,x_2)=&(0,x_1),\,\forall\,(x_1,x_2)\in H_2,\\
B_2(x_1,x_2)=&(x_1,2x_2),\,\forall\,(x_1,x_2)\in H_2.
\end{align*}
 Note that $B_2=I_{\{x_1\}}\oplus 2I_{\{x_2\}}.$ Observe that $B_1$ is nilpotent and $B_2$ is the direct sum of scalar multiple of unitary operators.
\end{example}
We deduce the following well known result from Theorem \ref{thm hypo char}.
\begin{corollary}
	Let $T\in\mathcal{K}(H)$ be a hyponormal operator. Then $T$ is normal
\end{corollary}
\begin{proof}
	If $T$ is compact, then $\lambda=0$ in Theorem \ref{thm hypo char}. Hence $S_2=0$. So $T=S_0=\underset{i=1}{\overset{n_0}{\oplus}}\lambda_iU_i$ for $n_0\in\mathbb{N}\cup\{\infty\}$, which is clearly a normal operator.
\end{proof}
\begin{corollary}
	If $T\in\mathcal{AN}(H)$ is a hyponormal operator with $\sigma_{ess}(T)=\omega(T)$, then $T$ must be normal.
\end{corollary}
\begin{proof}
	By Theorem \ref{thm hypo char}, we can write $T$ as
	\[T=
	\begin{blockarray}{cccc}
	H_0&H_1&H_2\\
	\begin{block}{(ccc)c}
	S_0&0&0&H_0\\
	0&\lambda S_1&A&H_1\\
	0&0&S_2&H_2\\
	\end{block}
	\end{blockarray},
	\]
	where $S_0,\,S_1,S_2,\,A$ and $\lambda$ are as defined in Theorem \ref{thm hypo char}.
	Note that
	\[
	T=
	\begin{bmatrix}
	S_0&0&0\\
	0&\lambda S_1&0\\
	0&0& 0
	\end{bmatrix}+
	\begin{bmatrix}
	0&0&0\\
	0&0&A\\
	0&0&0
	\end{bmatrix}+
	\begin{bmatrix}
	0&0&0\\
	0&0&0\\
	0&0&S_2
	\end{bmatrix}.
	\]
	Write
	\[
	\tilde{T}=
	\begin{bmatrix}
	S_0&0&0\\
	0&\lambda S_1&0\\
	0&0&0
	\end{bmatrix}.
	\]
	This implies that $\sigma_{ess}(\tilde{T})\subseteq\sigma(\tilde{T})\subseteq\sigma(S_0)\cup\sigma(\lambda S_1)\cup\{0\}$ and consequently  Area$(\sigma_{ess}(\tilde{T}))\leq\text{Area}(\sigma(S_0))+\text{Area}(\sigma(\lambda S_1))$ by \cite{PUTNAM}. Here Area$(A)$ is the planar Lebesgue measure of $A$.
	
	Since $A$ and $S_2$ are finite-rank operators, we have $\sigma_{ess}(T)=\sigma_{ess}(\tilde{T})$ and $\omega(T)=\omega(\tilde{T})$, so $\sigma_{ess}(\tilde{T})=\omega(\tilde{T})$. Also note that $\text{ind}(\tilde{T}-\delta I)=\text{ind}(\lambda S_1-\delta I_{H_1})$ for every $\delta\in\mathbb{C}$. We know that $\sigma_{ess}(\tilde{T})=\omega(\tilde{T})$ thus $\sigma_{ess}(\lambda S_1)=\omega(\lambda S_1)$. Since $S_0$ is the direct sum of scalar multiple of unitary operators, we have $\sigma_{ess}(S_0)=\omega(S_0)$. As both $S_0$ and $\lambda S_1$ are hyponormal operators, by \cite[Theorem 3.1]{COBURN} we have $\sigma(S_0)\setminus\omega(S_0)=\pi_{00}(S_0)$ and $\sigma(\lambda S_1)\setminus\omega(\lambda S_1)=\pi_{00}(\lambda S_1)$.
	
	Since $\sigma_{ess}(S_0),\,\sigma_{ess}(\lambda S_1)\subseteq C(0,\lambda)$, we have
	\begin{align*}
	\text{Area}(\sigma(S_0))=&\text{Area}(\omega(S_0))=\text{Area}(\sigma_{ess}(S_0))=0\text{ and }\\
	\text{Area}(\sigma(\lambda S_1))=&\text{Area}(\omega(\lambda S_1))=\text{Area}(\sigma_{ess}(\lambda S_1))=0.
	\end{align*}
	 Again by \cite[Theorem 3.1]{COBURN}, we know that $\sigma(T)\setminus\omega(T)=\pi_{00}(T)$ and it follows that
	 \begin{align*}
	 \text{Area}(\sigma(T))=\text{Area}(\omega(T))=\text{Area}(\sigma_{ess}(T))&=\text{Area}(\sigma_{ess}(\tilde{T}))\\
	 &\leq\text{Area}(\sigma(S_0))+\text{Area}(\sigma(\lambda S_1))\\
	 &=0.
	 \end{align*}
	Since $T$ is hyponormal operator, by \cite[Theorem 1]{PUTNAM} we have
	$$\|T^*T-TT^*\|\leq\frac{1}{\pi}\text{Area}(\sigma(T))=0.$$
	Hence $T$ is normal.
\end{proof}

In \cite[Question 3.12]{RAMpara}, it is asked that when a paranormal $\mathcal{AN}$-operator is normal? In the same paper, sufficient conditions are given for this problem. Here we prove that we can drop some conditions and still we get normality. It is to be mentioned that in \cite[Corollary 3]{Takeki} proved that the paranormality of $T$ and $T^*$ imply that $T$ is normal. Here we give a direct proof by assuming extra condition, namely the $\mathcal{AN}$-property.
\begin{theorem}\label{paranormalANisnormal}
	Let $T,T^*$ be paranormal and $T\in\mathcal{AN}(H)$. Then $T$ is normal.	
\end{theorem}
\begin{proof}

If $|T|$ has no eigenvalue with infinite multiplicity, then by \cite[Theorem 3.13]{RAMpara}, $T$ must be normal.

Next assume, that $|T|$ has eigenvalue with infinite multiplicity.
	By \cite[Theorem 3.9]{RAMpara}, $T$ can be represented as
	\begin{equation*}
	T= \left(\underset{\underset{\beta\ne \beta_0}{\beta \in \sigma(|T|)} }{\oplus} \beta U_{\beta}\right)\oplus \beta_0 U_{\beta_0},
	\end{equation*}
	where $\sigma_{ess}(|T|)={\{\beta_0}\}$. Here $U_{\beta}$ is unitary for all $\beta\ne \beta_0$.
	
	Since $|T|$ has an eigenvalue with infinite multiplicity, then it must be $\beta_0$.
	In this case  $U_{\beta_0}$ is an isometry and $H_{\beta_0}=N(|T|-\beta_0 I)$ is a reducing subspace for $T$. Moreover, $T|_{H_{\beta_0}}=\beta_0U_{\beta_0}$. To show $T$ is normal, it is enough to prove that $U_{\beta_0}$ is unitary. As $\beta_0 U_{\beta_0}=T|_{H_{\beta_0}}$ and the restriction of a paranormal operator to any invariant subspace is also paranormal, thus both the operators $U_{\beta_0}$ and $U_{\beta_0}^*$ are paranormal. By \cite[Theorem 1]{FURUTA}, we conclude that $U_{\beta_0}$ is unitary and hence $T$ is normal.
\end{proof}
\section*{acknowledgement}
We thank Professor Hiroyuki Osaka (Ritsumeikan University, BKC campus, Japan) for his suggestions which improved the clarity of the paper.

\end{document}